\documentclass{amsart}
\usepackage{amssymb}
\newcommand{\Da}{D}
\newcommand{\Db}{E}
\newcommand{\Dc}{F}
\newcommand{\Za}{X}
\newcommand{\Zb}{Y}
\newcommand{\Zc}{Z}
\newcommand{\Wa}{U}
\newcommand{\Wb}{V}
\newcommand{\Wc}{W}

\newcommand{\OO}{\mathbb O}
\newcommand{\KK}{\mathbb K}
\newcommand{\LL}{\mathbb L}

\newcommand{\so}{\mathfrak{so}}

\newcommand{\su}{\mathfrak{su}}
\newcommand{\ff}{\mathfrak{f}}
\newcommand{\ee}{\mathfrak{e}}
\newcommand{\lie}{\mathcal{L}}

\renewcommand{\cdot}{\circ}

\newcommand{\ad}{\mathrm{ad}\,}

\newtheorem{theorem}{Theorem}
\newtheorem{lemma}{Lemma}

\begin{document}

\title{
An octonionic construction of $E_8$ and the Lie algebra magic square}
\thanks{
This work was supported in part by grants from the John Templeton Foundation
and the Foundational Questions Institute (FQXi).}

\author{Robert A. Wilson \and Tevian Dray \and Corinne A. Manogue}
\address{R. A. Wilson \\
             School of Mathematical Sciences,
Queen Mary University of London,
London E1 4NS, UK.}
  \email{r.a.wilson@qmul.ac.uk}               
\address{           T. Dray \\
              Department of Mathematics,
Oregon State University, Corvallis, Oregon 97331, USA. }
  \email{tevian@math.oregonstate.edu} 
\address{C. A. Manogue \\
Department of Physics,
Oregon State University, Corvallis, Oregon 97331, USA.
}       
\email{corinne@physics.oregonstate.edu}

\date{10th April 2022}

\maketitle
\begin{abstract}
We give a new
construction of the Lie algebra of type $E_8$, 
in terms of $3\times3$ matrices, such that the Lie bracket has
a natural description as the matrix commutator.
This leads to a new interpretation of the Freudenthal--Tits magic square
of Lie algebras, acting on themselves by commutation. 
\end{abstract}

\centerline{\it Dedicated to the memory of Jacques Tits}

\section{Introduction}
\label{intro}
The Chevalley construction of $E_8$, as complex Lie algebra and complex Lie group,
generalizes easily to finite and algebraically closed fields
\cite{Chevalley}. However, over the real
numbers there are three different forms of $E_8$, known as the split, semi-split and
compact real forms. The Chevalley construction naturally produces the split form,
written as $248\times 248$ real matrices. By a straightforward 
change of basis, the other real
forms can be written as $248\times 248$ complex matrices, but writing them as
$248\times 248$ real matrices is more of a challenge. One approach to this problem
is to use the Freudenthal--Tits `magic square'.

According to Vinberg's interpretation \cite{Vinberg}
of the Freudenthal--Tits magic square,
the compact real form of the Lie algebra of type $E_8$ 
can be identified with an algebra of traceless anti-Hermitian $3\times 3$ matrices
over $\KK\otimes\LL$, plus the derivations of $\KK$ and $\LL$,
where $\KK\cong\LL\cong\OO$ is the division algebra of real octonions.
If $\KK\cong \OO$, but $\LL\cong\OO'$ is the algebra of spilt real octonions, then
the semi-split real form of $E_8$ is obtained instead, and if $\KK\cong\LL\cong\OO'$,
the split real form is obtained by the same process. 
However, 
the definition of the Lie bracket given by Vinberg
is not straightforward, and involves some quite complicated and
unintuitive formulae. In one sense, the complications arise because the 
subalgebra of type $D_4+D_4$ is divided up between the diagonal matrices
and the derivations.

Another interpretation, due to Barton and Sudbery \cite{BS}, 
unifies the subalgebra of type $D_4+D_4$
in what they call the triality construction,
but throws away the matrix interpretation. In this paper we combine
the best of the two approaches into a triality-invariant, $3\times 3$
matrix construction of $E_8$.
The key to our constructions is a simple octonionic notation for 
both the split and compact real forms of the Lie algebra of type $D_4$, 
which enables us to put these Lie algebras on the diagonal of $3\times 3$ matrices, leaving
the off-diagonal matrices to be anti-Hermitian as in the Vinberg interpretation. Then 
it turns out, somewhat surprisingly, that the
Lie bracket can be obtained as a natural extension of the usual matrix commutator.

The same ideas give rise to a simple construction of the magic square of
$2\times 2$ matrix algebras (see \cite{BS}), which we discuss first. 
In this case, the real Lie algebras which appear are all of the form $\so(p,q)$,
and the corresponding spin representations can be used to show that we actually construct associative algebras
whose Lie algebras are those which appear in the magic square. This approach is
fully discussed by Dray, Huerta and Kincaid \cite{DHK}.
Since these associative algebras are just algebras of matrices, with linear operators as 
matrix entries,
the Jacobi identity follows automatically.

The $2\times 2$ matrix algebras over $\OO$ and $\OO'$ are isomorphic to the
real Lie algebras $\so(9)$ and $\so(5,4)$ respectively.
Extending to $3\times 3$
matrices is then accomplished by introducing an explicit triality 
automorphism, which itself comes automatically from the exceptional Jordan algebra,
or Albert algebra,
in order to identify three copies of $\so(8)$
or $\so(4,4)$. This identification is in effect a
quotient of the vector space of matrices, which means that the matrices
themselves do not form an algebra (either associative or Lie).
However, 
the quotient space does act on the Albert algebra,
which implies that it forms a Lie algebra of type $\ff_4$ under commutation.
Finally we derive the adjoint representation, in order to facilitate
extending to $\ee_8$.

The most interesting case, of course, is the algebra of $3\times 3$ matrices over
$\mathbb O\otimes \mathbb O$ (resp. $\OO\otimes \OO'$, $\OO'\otimes \OO'$), 
which is expected to give compact (resp. semisplit, split) $\ee_8$.
The algebra we construct below
is certainly generated by two copies of $\ff_4$, and it is fairly easy
to see that the underlying vector space has dimension $248$. 
Since the algebra acts on this vector space, the Jacobi identity holds
automatically. The final piece of the argument,
therefore, is to show that the $248$-dimensional
space is closed under taking commutators.

In this paper, all algebras are real unless otherwise specified,
including octonion algebras and Lie algebras, and Jordan algebras.
Our constructions work equally well
for both compact and split versions of the octonions (and quaternions,
and complex numbers), so give explicit versions of all 
the real forms of $\ee_8$.

In a subsequent paper 
we show how to restrict the
construction of $\ee_8$ to give constructions of the entire magic square,
first in the adjoint representations.
This is not as straightforward as in the case of $2\times 2$ matrices,
where the spin representations permit direct construction of each
entry in the magic square.
In particular, the constructions of $\ee_6$ and $\ee_7$ exhibit interesting features,
and permit construction not only of the adjoint representations, but also
of the minimal representations.

The paper is organised as follows. In Section~\ref{octo}
we recall necessary background material on octonions and $\so(8)$,
including triality. In Section~\ref{so9} we analyse the $2\times 2$
octonion matrix algebra, and the related Lie algebra $\so(9)$.
This material is then extended in two independent ways. First, in Section~\ref{so16},
we extend to the tensor product of two octonion algebras, in order to
construct $\so(16)$. Second, in Section~\ref{f4}, we extend to
$3\times 3$ matrices, in order to construct the Albert algebra and the
related Lie algebra $\ff_4$. Finally, in Section~\ref{e8}, we make both
extensions at the same time, to construct $\ee_8$.

\section{Octonions and $\so(8)$}
\label{octo}
\subsection{Octonions and split octonions}
The (real) 
octonions are an $8$-dimensional division algebra $\mathbb O$, with (real) basis
$1$, $i$, $j$, $k$, $\ell$, $i\ell$, $j\ell$, $k\ell$, with the property that all the $7$ imaginary basis units
square to $-1$, any two associate but anti-commute, $ij=-ji=k$, and 
any three whose product is not $\pm1$
anti-associate,
in the sense that, for example, $(ij)\ell=-i(j\ell)$. From this the entire multiplication table can be
computed. There is an involution $x\mapsto \overline{x}$ which fixes $1$
and negates all the other basic units, and a (squared) norm $N(x)=x\overline{x}$.
There is also a $7$-fold symmetry which cycles $(i,j,\ell,k,j\ell,-k\ell,i\ell)$,
and a $3$-fold symmetry $(i,j,k)(i\ell,j\ell,k\ell)$. Moreover, if $i,j,k$ are fixed,
then any of $i\ell,j\ell,k\ell$ can be used in place of $\ell$.

The split octonions $\mathbb O'$ are similar, except that four of the imaginary units
square to $+1$ instead of $-1$. For example, we may take a basis conisting of
$1,I,J,K,L,IL,JL,KL$,
and decree that $L,IL,JL,KL$ square to $+1$. We lose the $7$-fold symmetry, of course,
but keep the $3$-fold symmetry.
It is also still true that any of $IL$, $JL$, or $KL$ can be used in place of $L$.

In what follows, we shall notate two copies of the octonions by lower-case
and upper-case letters, and usually (but not always) 
assume that $i,j,k,I,J,K$ square to $-1$,
and explicitly state $\ell^2=1$ or $\ell^2=-1$, and $L^2=1$ or $L^2=-1$, in all
cases where it actually matters. 
In general, we prove results only for the division algebra $\OO$, and note that the
same proof works for $\OO'$, with a few signs changed.

Many results do not depend on the actual basis, and
will therefore be stated for more general $p,q,r,s$, which will usually denote arbitrary
mutually orthogonal
pure imaginary octonions. However, in most cases there is
no loss of generality in assuming that
$p,q,r,s$ belong to the standard basis. Similarly, $a,b,c$ will often denote arbitrary
octonions, not assumed imaginary, and not assumed to have norm $\pm1$.

\subsection{The $\so(8)$ Lie algebra}
The (compact, real) $\so(8)$ Lie algebra may be generated by the left actions of the octonions
on themselves. As a vector space this algebra is then spanned by the $7$ 
single left multiplications $L_i$, \ldots $L_{k\ell}$ and the $21$ `nested' multiplications 
$L_i\circ L_j$
(meaning $L_j$ followed by $L_i$), \ldots, $L_{j\ell}\circ L_{k\ell}$.
In the associative algebra generated by these operations,  $L_p\circ L_q=-L_q\circ L_p$,
and $L_p\circ L_p=L_{p^2}$, and
we easily deduce
the following rules,
where $p,q,r,s$ are (arbitrary) orthogonal imaginary octonions:
\begin{eqnarray}
L_p\circ (L_p\circ L_q) &=& - (L_p\circ L_q)\circ L_p\cr &=& p^2 L_q\cr
(L_r\circ L_p)\circ (L_p\circ L_q) &=& - (L_p\circ L_q)\circ (L_r\circ L_p)\cr 
&=& p^2 (L_r\circ L_q)\cr
L_r\circ (L_p\circ L_q)&=& (L_p\circ L_q)\circ L_r\cr
(L_r\circ L_s)\circ (L_p\circ L_q) &=& (L_p\circ L_q)\circ (L_r\circ L_s)
\end{eqnarray}
Where it does not cause confusion, we shall write $fg$ for $f\circ g$ and
$[f,g]$ for $fg-gf$.
Hence in the Lie algebra we have 
\begin{eqnarray}
[L_p,L_q]&=&2L_pL_q\cr
[L_p,L_p L_q]&=& 2p^2 L_q\cr
[L_r L_p,L_p L_q]&=& 2p^2 L_rL_q\cr
[L_r,L_pL_q]&=&0\cr
[L_p L_q,L_rL_s]&=&0.
\end{eqnarray}
In the compact case, where $p^2=-1$ for all imaginary elements of norm $1$, 
the map taking $L_p$ to $1\wedge p$, and $L_{p}L_{q}$ to $p\wedge q$, for $p,q$
imaginary elements of the standard basis,
shows the isomorphism with the standard construction of (the
compact real form of) $\so(8)$. 
Here the Lie bracket is given by
$$[p\wedge q,p\wedge r]=2q\wedge r,$$
where we now also allow the variables $p,q,r$ to be real.
If we use instead the split octonions,
we obtain a different real form, namely the split real form
$\so(4,4)$.

\subsection{Right actions}
For any $a\in\OO$, define left, right, and double actions of $a$ on $\OO$ by
\begin{eqnarray}
L_a&:& x\mapsto ax\cr
R_a&:& x\mapsto xa\cr
B_a:=-L_a-R_a&:& x\mapsto -ax -xa
\end{eqnarray}
As is well-known, the Lie algebra $\so(8)$ is generated by any one
of the left, right, or double actions of the imaginary part of
$\OO$ on $\OO$. 
To see this explicitly, we may restrict to the standard basis of $\OO$, 
and then by symmetry it
suffices to verify that
\begin{eqnarray}
L_\ell&=&(-R_\ell+R_iR_{i\ell}+R_jR_{j\ell}+R_kR_{k\ell})/2
\end{eqnarray}
This is an easy exercise.
(Here we use the natural convention for functions, 
that $R_iR_j$ means $R_j$ followed by $R_i$, rather than
the natural convention for right multiplications, that $R_iR_j$
means $R_i$ followed by $R_j$.)
To obtain the full set of such equations, apply the $7$-cycle  $(i,j,\ell,k,j\ell,-k\ell,i\ell)$ to the indices.
For convenience, we record the following useful identities.
\begin{eqnarray}
L_iL_{i\ell}&=&(-R_\ell+R_iR_{i\ell}-R_jR_{j\ell}-R_kR_{k\ell})/2\cr
L_i&=&(-R_i+R_{i\ell}R_\ell+R_kR_j+R_{j\ell}R_{k\ell})/2
\end{eqnarray}
Triality cycles $L_p$ to $R_p$ to $B_p$, while duality swaps $-L_p$ with $R_p$, and therefore
negates $B_p$. Hence we obtain corresponding formulae expressing each of the operators
of type $L$, $R$ or $B$ in terms of either of the other two types.

Analogous formulae hold in the split octonions, sometimes with different signs. The rule
for signs is that whenever $\ell$ appears twice in a single term, the corresponding term is negated 
in the split octonions relative
to the compact octonions. For example, in the split octonions
\begin{eqnarray}
L_I&=&(-R_I-R_{IL}R_L+R_KR_J-R_{JL}R_{KL})/2
\end{eqnarray}

\section{The Lie algebra $\so(9)=\su(2,\OO)$}
\label{so9}
\subsection{The spin representation}
We are now ready to start looking at the $n=2$ magic square.
We start the investigation with the $\so(9)$ case, that is the top-right
or bottom-left corner, as that's
analogous to $\ff_4$ in the $n=3$ square. 
This algebra is already well-understood as an algebra of $2\times2$ octonion matrices
(see  Manogue and Schray \cite{ManSchray}), and what we give here is
just a new notation, and a slightly different interpretation.

Our approach is 
to start with the
off-diagonal matrices, and use them to construct the rest of the
algebra. We simply take them
to be anti-Hermitian matrices in the usual sense, spanned by
\begin{eqnarray}
\Za_1:=\begin{pmatrix}0&1\cr-1&0\end{pmatrix} &\qquad&\Za_p:=\begin{pmatrix}0&p\cr p&0\end{pmatrix}
\end{eqnarray}
where $p$ can be any pure imaginary octonion.
We let these matrices act on $2$-component columns (of octonions) 
in the obvious way. Since these actions, regarded as maps on real $16$-space,
are linear, they generate an associative algebra.
It follows that under commutation they generate a Lie algebra. 

On the other hand, the octonions are not associative, so the commutator of
two actions, each expressed as a $2\times 2$ octonionic
matrix, \emph{cannot} in general be expressed
as the action of the commutator of the matrices. This forces us to incorporate into
the matrices a notation for a commutator of left actions of octonions.
Let us write $p\circ q=[L_p,L_q]/2$ 
for this kind of
`product' of $p$ and $q$. 
With this definition, we obtain the following diagonal matrices, considered as left actions
on octonionic (column) $2$-vectors, where $p$ and $q$ are orthogonal imaginary octonions:
\begin{eqnarray}
\Da_p=\Da_{1,p}=[\Za_1,\Za_p]/2&=&\begin{pmatrix}p&0\cr 0&-p\end{pmatrix}\cr
\Da_{p,q}=-\Da_{q,p}=[\Za_{p},\Za_{q}]/2&=&\begin{pmatrix}{p\circ q}&0\cr 0&
{p\circ q}\end{pmatrix}
\end{eqnarray}

It is now clear that the matrices $\Da_p$ and $\Da_{p,q}$,
acting on $2$-component column vectors,  
satisfy the same commutation relations as the
operators $L_p$ and $L_pL_q$,  
acting on octonions. In particular they
are closed under commutation.
Therefore the matrices $D_p$ and $D_{p,q}$ give a representation of $\so(8)$. 
The first octonion coordinate gives the half-spin representation defined by left
octonion multiplication. The second octonion coordinate is equivalent to the
other half-spin representation, which is defined here by left-multiplication
by the octonion conjugate, but which can equivalently be defined by right-multiplication
by the octonions themselves.
\subsection{The Lie algebra}
We claim that this $\so(8)$ together with
$\Za_1$ and the $\Za_p$, for $p$ any imaginary basic unit, generate a copy of
$\so(9)$. 
In order to prove this fact, 
it is necessary first to check closure under commutation. 
Since the commutators of two $\Za$s or two $\Da$s have already been considered,
all that remains 
is to check the commutators of the actions of the $\Za_a$ with the actions of $\Da_p$
and $\Da_{p,q}$.
This is almost a triviality: we compute the following
\begin{eqnarray}
[\Za_1,\Da_p]&=&
-2\Za_p\cr
[\Za_p,\Da_p]&=&
-2p^2\Za_1
\cr
[\Za_q,\Da_p]&=&
0\cr
[\Za_1,\Da_{p,q}]&=&
0\cr
[\Za_p,\Da_{p,q}]&=&
2p^2\Za_q\cr
[\Za_{r},\Da_{p,q}]&=&
0
\end{eqnarray}
Notice that the first three rules can be expressed in a single equation as
\begin{eqnarray}
[\Da_p,\Za_a]&=&\Za_{ap}+\Za_{pa}
\end{eqnarray}
which will be useful in later calculations.
Moreover, the last three of the six rules
follow immediately from the first three together with
the Jacobi identity for matrix actions, thus: 
\begin{eqnarray}
[\Za_1,\Da_{p,q}] &=&[\Za_1,[\Za_p,\Za_q]]/2\cr
&=&[\Za_p,[\Za_1,\Za_q]]/2 - [\Za_q,[\Za_1,\Za_p]]/2\cr
&=&[\Za_p,\Da_q]/2-[\Za_q,\Da_p]/2\cr
&=&0\cr
[\Za_p,\Da_{p,q}]&=&[\Za_p,[\Da_p,\Da_q]]/2\cr
&=&-[\Da_p,[\Da_q,\Za_p]]/2 - [\Da_q,[\Za_p,\Da_p]]/2\cr
&=&p^2[\Da_q,\Za_1]\cr
&=&2p^2\Za_q\cr
2[\Za_r,\Da_{p,q}]&=&[\Za_r,[\Da_p,\Da_q]]\cr
&=&-[\Da_p,[\Da_q,\Za_r]] - [\Da_q,[\Za_r,\Da_p]]\cr
&=&0
\end{eqnarray}

\subsection{The adjoint representation}
By this stage we know we have a Lie algebra of dimension $36$, containing
$\so(8)$. It is easy to see that it is simple, and that it is a copy of $\so(9)$.
More explicitly, an isomorphism with the standard copy of $\so(9)$ in its adjoint
representation can be obtained as follows.
If we label the $9$ coordinates for $\so(9)$ as $1',1,i,j,k,\ell,i\ell,j\ell,k\ell$ then the
isomorphism with the usual Lie algebra, namely the skew square of the
natural representation, is given by
\begin{eqnarray}
\Za_{a}&\mapsto&1'\wedge a\cr
\Da_p&\mapsto&1\wedge p\cr
\Da_{p,q}&\mapsto&p\wedge q
\end{eqnarray}
where $a,p,q$ are elements of the standard basis of $\OO$, and $p,q$ are imaginary.
Here the Lie bracket is given by
$$[a\wedge b,a\wedge c]= 2[b\wedge c]$$
and other products are zero. 

Strictly speaking, the adjoint representaton $\ad$ is defined by 
$$\ad x : y \mapsto [x,y],$$
for every $x,y$ in the algebra.
Thus $\ad x$ is a linear map on the underlying vector space of the algebra.
This somewhat pedantic distinction between the algebra itself and its adjoint representation
becomes helpful later on when we come to consider the $\ff_4$ algebra and
its representation  on the Albert algebra, and even more so when we come
to build $\ee_8$.

The representation of the algebra on $2$-component octonion
columns 
is of course
the spin representation of $\so(9)$, of real dimension $16$.

\subsection{The natural representation}
\label{natso9}
We show how the natural representation of $\so(9)$ can also be constructed
using $2\times 2$ octonion matrices. This will be required
in the construction of $\ff_4$ below, in its action on the Albert algebra.

Let $\mathcal U$ be the space of
$2\times 2$ trace $0$ Hermitian matrices, spanned by 
\begin{eqnarray}
\Wa_{1'}=\begin{pmatrix}1&0\cr 0&-1\end{pmatrix}& \mbox{and}& 
\Wa_a=\begin{pmatrix}0&a\cr\overline{a}&0\end{pmatrix} \mbox{ for } a=1,i,j,k,
\ell,i\ell,j\ell,k\ell.
\end{eqnarray}
We define the action of the $\Za_b$ on $\mathcal U$ via
\begin{eqnarray}
\Za_b&:&\Wa_a\mapsto \Za_b\Wa_a-\Wa_a\Za_b
\end{eqnarray}
and check that $\Za_b$ maps $\Wa_{1'}$ to $-2\Wa_b$, 
and $\Wa_b$ to $2\Wa_{1'}$, while sending
all other basis elements $\Wa_a$ to $0$. Since this is (up to an overall scalar factor) 
the natural action of the $\Za_b$ on a $9$-dimensional
orthogonal space, it extends to a representation of the whole of $\so(9)$, provided we
define the actions of the commutators to be the commutators of the actions.

In the case $\Da_p=[\Za_1,\Za_p]/2$ there is enough associativity to ensure that the
action of $\Da_p$ is the same as the obvious action of the matrix
$\begin{pmatrix}p&0\cr0&-p\end{pmatrix}$. In the case $\Da_{p,q}$ we compute the
action as $\Wa_p\mapsto \Wa_q\mapsto-\Wa_p$, and every other $\Wa_a$ maps to $0$.
Now let us compare this with the action of a hypothetical `matrix'
$\begin{pmatrix}p\circ q&0\cr 0&p\circ q\end{pmatrix}$. To obtain the correct action
on $\mathcal U$, it is necessary and sufficient to interpret the commutators as
\begin{eqnarray}
[p\circ q,a] &=&p\circ q\circ a-a\circ p\circ q 
\end{eqnarray} 
since this evaluates to $0$ except when $a=p$ or $a=q$, when it evaluates to
$2q$ and $-2p$ respectively. (If $a=1$ we interpret $p\circ q\circ a=a\circ p \circ q=p\circ q$.)

\subsection{Matrix commutators}
\label{so9mat}
When working with matrices over an associative algebra, the
Lie bracket automatically
agrees with the commutator of matrices. Over octonions, one cannot
expect this. However, it is remarkable how close we can get to this aim in this case.
 The commutators of the $\Za$s were used to define the $\Da$s.
The commutators of the $\Da$s have already been shown to follow the same rules
as the $L$s. These rules explain how to interpret commutators of
`nested' octonions.

It remains to consider 
the commutators of the $\Za$s with the $\Da$s. These contain off-diagonal entries which
are either of the form $ap+pa$ (in the case of $[\Za_a,\Da_p]$), or of the form
$a(p\circ q)-(p\circ q) a$ (in the case of $[\Za_a,\Da_{p,q}]$).
The former gives the correct commutator as a matrix if the products are just multiplied out
as octonions. The latter also gives the correct commutator provided the given entry
is evaluated as $[a,p\circ q]$ using the same rules as for the diagonal entries.
However, it does not make sense to try to evaluate the individual terms $a(p\circ q)$
or $(p\circ q)a$ separately.

This interpretation now gives an \emph{abstract} construction of the Lie algebra,
without reference to its action on column vectors. It consists of anti-Hermitian matrices,
in the extended sense which includes `nested' diagonal matrices, and the Lie bracket
is just the matrix commutator, extended by a more-or-less obvious interpretation
of commutators involving nested octonions. Since this result is by no means obvious,
it is perhaps worth stating formally.

\begin{theorem}\label{so9thm}
Let $\lie$ be the real $36$-space spanned by  the $2\times 2$ octonionic matrices
$$\Za_a=\begin{pmatrix}0&a\cr-\overline{a}&0\end{pmatrix},
\Da_p=\begin{pmatrix}p&0\cr0&-p\end{pmatrix},
\Da_{p,q}=\begin{pmatrix}p\circ q&0\cr 0&p\circ q\end{pmatrix}
$$
where $p,q$ are any distinct imaginary basic units, and $a$ is any basic octonionic unit.
Define a commutator on $\lie$ as the matrix commutator, subject to 
evaluating all entries as left-multiplications. This is equivalent to
the
following special rules:
\begin{enumerate}
\item Diagonal entries are evaluated as normal
except that $pq-qp$ is evaluated as $2p\circ q$, so that $[\Za_p,\Za_q]=2\Da_{p,q}$.
\item Off-diagonal entries $ab+ba$ are evaluated as normal octonions, so
are $0$ on the basis unless $a=1$ or $b=1$ or $a=b$.
\item Off-diagonal entries $[p\circ q,r]$ are evaluated as $p\circ q\circ r-r\circ p\circ q$,
so are $0$ on the basis unless $r=p$ or $r=q$.
\end{enumerate}
Then $\lie$ with this commutator map is isomorphic to the Lie algebra $\so(9)$, if the
octonion division algebra is used, or to $\so(5,4)$, if the split octonions are used. 
\end{theorem}

\section{The 
Lie algebra $\so(16)=\su(2,\OO\otimes\OO)$}
\label{so16}
\subsection{The spin representation}
The same idea enables us to construct $\so(16)$ by taking matrices defined over
the tensor product of two copies of the octonions.
The hard work, such as it was, has already been done.
Let $\KK$ and $\LL$ be two copies of the octonions.
This time, we let the algebra act on two-component vectors with entries
in $\KK\otimes \LL$. This vector space has real dimension $128$, and
affords the spin representation of $\so(16)$.
(See \cite{DHK}.)

The algebra may be generated by two copies of $\so(9)$, that is $\su(2,\mathbb K)$
and $\su(2,\mathbb L)$ as constructed in Section~\ref{so9}.
These each act in the obvious way on one component of the tensor products, and intersect in
a copy of $\so(2)$ spanned by $\Za_1$. The diagonal subalgebras
$\so(\mathbb K)\cong\so(8)$ and $\so(\mathbb L)\cong\so(8)$ commute with each other,
by construction.

We already have $15$ dimensions of off-diagonal matrices, and there are $49$ more
of type 
\begin{eqnarray}
\Za_{pP}:=
[\Da_p,\Za_P]/2&=&\begin{pmatrix}0&pP\cr -pP & 0\end{pmatrix},
\end{eqnarray}
making $64$ in total.
This makes up the full $28+28+64=120$ dimensions of $\so(16)$.
It is straightforward to show that this $120$-space is closed under commutation, and is a
compact simple Lie algebra containing $\so(8)+\so(8)$, so is $\so(16)$.
However, we shall give a more explicit isomorphism with the standard copy of
$\so(16)$ in Section~\ref{Lieso16} below.

\subsection{The Lie bracket}
\label{Lieso16}
It is easy to see that the commutators of the $\Da$s and $\Za$s are given by the same formulae
as before, with the extra subscript on the $\Za$s being carried through. 
More explicitly, we have
\begin{eqnarray}
[\Da_p,\Za_{aA}] &=& \Za_{apA}+\Za_{paA}\cr
[\Da_{p,q},\Za_{qA}] &=& 2q^2\Za_{pA}\cr
[\Da_{p,q},\Za_{aA}] &=& 0\mbox{ for } a\perp p,q
\end{eqnarray}
and the corresponding equations with the roles of the two alphabets interchanged.

It remains therefore to consider the commutators of two $\Za$s.
(As usual matrices should be understood to act on $2$-component
columns.) Assume that $a,b$ are orthogonal basic units in $\KK$ and $A,B$ are orthogonal
basic units in $\LL$. Then
\begin{eqnarray*}
[\Za_{aA},\Za_{bB}]&=&\begin{pmatrix}0&a A\cr -\overline{a}\overline{A}&0\end{pmatrix}
\begin{pmatrix}0&b B\cr -\overline{b}\overline{B}&0\end{pmatrix}
-
\begin{pmatrix}0&b B\cr -\overline{b}\overline{B}&0\end{pmatrix}
\begin{pmatrix}0&a A\cr -\overline{a}\overline{A}&0\end{pmatrix}
\end{eqnarray*}
which evaluates to
\begin{eqnarray*}
&&\begin{pmatrix}-(a\cdot\overline{b})(A\cdot\overline{B})+(b\cdot\overline{a})(B\cdot\overline{A})&0\cr
0&-(\overline{a}\cdot b)(\overline{A}\cdot B)+(\overline{b}\cdot a)(\overline{B}\cdot A)\end{pmatrix}
\end{eqnarray*}
We may assume that $a,b,A,B$ belong to the standard basis, and that
either $a\ne b$ or $A\ne B$. The arguments in the two cases are the same,
so without loss we may assume $A\ne B$.
First disjoin cases for $a$ and $b$.
In the case $a=b$, so that $a\cdot\overline{b}=b\cdot\overline{a}=\pm1$,
the top-left entry reduces to
$a\overline{a}(B\cdot\overline{A}-A\cdot\overline{B})$, while if $a\ne b$ it reduces to
$(b\cdot\overline{a})(B\cdot\overline{A}+A\cdot\overline{B})$.
Now disjoin cases for $A$ and $B$.
In the case $A=1$ and $B$ is imaginary, the former reduces to $2a\overline{a}B$ 
and the latter to $0$.
Similarly, if both $A$ and $B$ are imaginary, then the former reduces to $2a\overline{a}A\cdot B$,
and the latter to $0$ again. 
We find that the commutators are zero except for
\begin{eqnarray}
[\Za_{aA},\Za_{aB}]&=&2a\overline{a}\Da_{A,B}\cr
[\Za_{aA},\Za_{bA}]&=&2A\overline{A}\Da_{a,b}
\end{eqnarray}

To see the isomorphism with $\so(16)$ a little more explicitly, we may imagine $16$ 
coordinates labelled $1,i,\ldots,k\ell,1',I,\ldots,KL$ and identify
\begin{eqnarray}
\Da_i&=&1\wedge i\cr
\Da_{i,j}&=& i\wedge j\cr
\Da_{I}&=& 1'\wedge I\cr
\Da_{I,J}&=& I\wedge J\cr
\Za_{aA}&=& a\wedge A
\end{eqnarray} 
Then it is easy to check that all the non-zero commutator relations are
of the form $$[x\wedge y, x\wedge z]=2y\wedge z.$$

\subsection{Matrix commutators}
Again, the above computations are all done with actions on $2$-component column vectors.
However, just as in Section~\ref{so9}
above, we observe \emph{post hoc} that matrix commutators,
computed with the same rules as before, give the same answers. In other words,
we may define an abstract Lie algebra of type $\so(16)$, spanned by these matrices,
and with the Lie bracket defined by the matrix commutator, suitably interpreted.

Indeed, the only new commutators we need to consider are the $[\Za_{aA},\Za_{bB}]$.
But the computation above works exactly the same whether we work with actions,
or with abstract matrices (i.e. multiplying out $a\cdot b$ and $A\cdot B$ as octonion products
wherever they arise). 
The reason for this is that the only rules we are using are commutation rules between
octonion actions, and these are the same as the corresponding rules for
octonion multiplications: essentially just that $L_p\circ L_q=-L_q\circ L_p$ corresponds to
$pq=-qp$.
Hence no new rules of matrix multiplication are required.

For completeness we write out this result as a formal theorem.

\begin{theorem}\label{so16thm}
Let $\lie$ be the real $120$-space spanned by  the following $2\times 2$  matrices
over $\KK\otimes\LL$, where $\KK$ and $\LL$ are octonion algebras:
\begin{eqnarray}
\Za_{aA}=\begin{pmatrix}0&aA\cr-\overline{a}\overline{A}&0\end{pmatrix},
&\Da_p=\begin{pmatrix}p&0\cr0&-p\end{pmatrix},&
\Da_{p,q}=\begin{pmatrix}p\circ q&0\cr 0&p\circ q\end{pmatrix}
\cr
&\Da_P=\begin{pmatrix}P&0\cr0&-P\end{pmatrix},&
\Da_{P,Q}=\begin{pmatrix}P\circ Q&0\cr 0&P\circ Q\end{pmatrix}
\end{eqnarray}
where $p,q$ are any distinct imaginary basic units of $\KK$, and $a$ is any basic unit,
and similarly for $P,Q,A\in\LL$.
Define a commutator on $\lie$ as the matrix commutator, subject to 
evaluating all entries as left-multiplications.
Then $\lie$ with this commutator map is isomorphic to the Lie algebra $\so(16)$, if 
both $\KK$ and $\LL$ are division algebras, or to $\so(8,8)$ if both are split octonion algebras,
 or to $\so(12,4)$, 
if they are one of each type.
\end{theorem}

\section{The Lie algebra $\ff_4=\su(3,\OO)$ 
}
\label{f4}
\subsection{Introduction}
The Lie algebra $\ff_4$ may be defined abstractly as the algebra of derivations 
on the exceptional Jordan algebra, 
or Albert algebra, that is an
algebra of $3\times 3$ Hermitian matrices over $\OO$.
However, this does not lead easily to a nice description of $\ff_4$, as there is no
obvious canonical parametrisation of the derivations. In effect, each derivation
has many names. 

A construction of the Lie \emph{group}
$F_4$, acting as automorphisms of the Albert algebra, 
 is described by Dray and Manogue \cite{DrayMan}, in terms of $3\times 3$
anti-Hermitian octonion matrices.
Indeed, generators for this group are already given by Jacobson in 1960 \cite{Jac2}.
Naively, the space of anti-Hermitian matrices has dimension $45$, and there are
significant technical difficulties in turning this into a workable description
of $52$-dimensional $F_4$. In particular, the elements of the group generated by the actions of the 
diagonal matrices do not have an obvious
canonical notation.
In a sense, it is straightforward to adapt these constructions
to construct the corresponding action of the
Lie algebra $\ff_4$ on the Albert algebra, in effect by differentiating the
group generators. But the problem of non-canonical names for automorphisms
and derivations is one that remains 
to be solved at some point.

In order to prepare the ground for our eventual construction of $\ee_8$,
we shall take a concrete `first-principles' approach, and build $\ff_4$ explicitly
from three copies of $\so(9)$, as constructed in Section~\ref{so9}.
We shall not use the Albert algebra structure directly, so we shall not prove directly that
the elements of our Lie algebra are derivations. Instead we shall prove that
our Lie algebra is simple, and compact, and has dimension $52$, so it follows
from the classification
that it is the compact real form of $\ff_4$.

\subsection{The action of $\so(9)$ on the Albert algebra}
We construct $\ff_4$ from the union of three copies of $\so(9)$, each acting on
two of the three coordinates. Thus
we first define the three types of off-diagonal anti-Hermitian matrices as
\begin{eqnarray}
\Za_a=\begin{pmatrix}0&a&0\cr -\overline{a}&0&0\cr 0&0&0\end{pmatrix},&
\Zb_a=\begin{pmatrix}0&0&0\cr0&0&a\cr0&-\overline{a}&0\end{pmatrix},&
\Zc_a=\begin{pmatrix}0&0&-\overline{a}\cr 0&0&0\cr a&0&0\end{pmatrix}.
\end{eqnarray}
The action of such a matrix $A$ on an element $H$ of the Albert algebra
is given by
\begin{eqnarray}
A&:&H\mapsto AH+H\overline{A}^\top=AH-HA
\end{eqnarray}
where $\overline{A}^\top$ denotes the octonion-conjugate transpose matrix,
and $\overline{A}^\top =-A$ and $\overline{H}^\top=H$. 

Writing 
\begin{eqnarray}
H&=&\begin{pmatrix} \rho&\gamma&\overline{\beta}\cr \overline{\gamma} & \sigma & \alpha\cr
\beta & \overline{\alpha} & \tau\end{pmatrix}
\end{eqnarray}
we compute
\begin{eqnarray}
[\Za_a,H]&=&\begin{pmatrix} a\overline{\gamma}+\gamma\overline{a} & a(\sigma-\rho) & a\alpha\cr
\overline{a}(\sigma-\rho) & -\overline{a}\gamma-\overline{\gamma}a & 
-\overline{a}(\overline{\beta})\cr
\overline{\alpha}(\overline{a}) & -\beta a & 0\end{pmatrix}
\end{eqnarray}
Hence
we see that the $\Za_a$ generate a copy of $\so(9)$, acting on the left on
$\begin{pmatrix}\overline{\beta}\cr\alpha\end{pmatrix}$ as the spin representation.
and on the right on $\begin{pmatrix}\beta & \overline{\alpha}\end{pmatrix}$
as the dual copy of the spin representation, and on both sides on
$\begin{pmatrix}(\rho-\sigma)/2&\gamma\cr \overline{\gamma} & (\sigma-\rho)/2\end{pmatrix}$ as the
natural representation. This leaves two copies of the trivial representation,
spanned by $\tau$ and $\begin{pmatrix}(\rho+\sigma)/2&0\cr 0&(\rho+\sigma)/2\end{pmatrix}$.

Thus we obtain the following diagonal matrices as commutators, which can be computed
directly with the matrices, just as in Section~\ref{so9mat}.
\begin{eqnarray}
\Da_p=[\Za_1,\Za_p]/2&=&\begin{pmatrix}p&0&0\cr 0&-p&0\cr 0&0&0\end{pmatrix}\cr
\Db_p=[\Zb_1,\Zb_p]/2&=&\begin{pmatrix}0&0&0\cr 0&p&0\cr 0&0&-p\end{pmatrix}\cr
\Dc_p=[\Zc_1,\Zc_p]/2&=&\begin{pmatrix}-p&0&0\cr 0&0&0\cr 0&0&p\end{pmatrix}\cr
\Da_{p,q}=[\Za_{p},\Za_{q}]/2 &=&\begin{pmatrix}{p\circ q}&0&0\cr 0&{p\circ q}&0\cr 0&0&0\end{pmatrix}\cr
\Db_{p,q}=[\Zb_{p},\Zb_{q}]/2 &=&\begin{pmatrix}0&0&0\cr 0&{p\circ q}&0\cr 0&0&p\circ q\end{pmatrix}\cr
\Dc_{p,q}=[\Zc_{p},\Zc_{q}]/2 &=&\begin{pmatrix}{p\circ q}&0&0\cr 0&0&0\cr 0&0&p\circ q\end{pmatrix}
\end{eqnarray}

We now have three copies of $\so(9)$, generated by 
\begin{enumerate}
\item the $\Da_p$, $\Da_{p,q}$ and $\Za_a$;
\item the $\Db_p$, $\Db_{p,q}$ and $\Zb_a$; and 
\item the $\Dc_p$, $\Dc_{p,q}$ and $\Zc_a$.
\end{enumerate}
All commutators within one of the three copies of $\so(9)$ can be computed by matrix
operations as in Section~\ref{so9mat}
above. Before we compute the rest of the commutators, we investigate
the intersection of the three copies of $\so(9)$. 

\subsection{Triality}
\label{f4triality}
The trace $0$ part of the Albert algebra is spanned by the elements
\begin{eqnarray}
\Wa_0=\begin{pmatrix}1&0&0\cr 0&-1&0\cr 0&0&0\end{pmatrix}&
\qquad&
\Wa_a=\begin{pmatrix}0&a&0\cr \overline{a}&0&0\cr 0&0&0\end{pmatrix}\cr
\Wb_0=\begin{pmatrix}0&0&0\cr0&1&0\cr0&0&-1\end{pmatrix}&
&\Wb_a=\begin{pmatrix}0&0&0\cr 0&0&a\cr 0&\overline{a}&0\end{pmatrix}\cr
\Wc_0=\begin{pmatrix}-1&0&0\cr 0&0&0\cr 0&0&1\end{pmatrix}&
&\Wc_a=\begin{pmatrix}0&0&\overline{a}\cr 0&0&0\cr a&0&0\end{pmatrix}
\end{eqnarray}
We compute the action of the diagonal matrices $\Da_p$, $\Db_p$ and $\Dc_p$ on the Albert
algebra as follows:
\begin{eqnarray}
\Da_p&:& \Wa_a \mapsto \Wa_{pa}+\Wa_{ap}\cr
\Db_p&:& \Wa_a \mapsto -\Wa_{ap}\cr
\Dc_p&:& \Wa_a \mapsto -\Wa_{pa}
\end{eqnarray}
as well as $\Wa_0\mapsto 0$,
and corresponding actions on $\Wb_a$ and $\Wc_a$.

It then follows
that these actions of $\Da_p$, $\Db_p$ and $\Dc_p$ on $\Wa_a$
are given by the negatives of the actions of $B_p$, $R_p$ and $L_p$
on the subscript $a$, in the three cases.
We have already described the triality relations which hold between these three
types of operators on the octonions, and it follows that the same relations hold
between the actions of the $\Da$s, $\Db$s and $\Dc$s. Specifically, 
this induces the following equivalence between matrices:
\begin{eqnarray}
\Db_\ell&\equiv& (-\Da_\ell-\Da_{i,i\ell}-\Da_{j,j\ell}-\Da_{k,k\ell})/2\cr
\Dc_\ell&\equiv& (-\Da_\ell+\Da_{i,i\ell}+\Da_{j,j\ell}+\Da_{k,k\ell})/2
\end{eqnarray}
These formulae 
appear for
example in Wangberg's PhD thesis \cite{Wangberg}.
From them one easily deduces the equivalent formulae for the remaining elements:
\begin{eqnarray}
\Db_{i,i\ell}&\equiv& (\Da_\ell+\Da_{i,i\ell}-\Da_{j,j\ell}-\Da_{k,k\ell})/2\cr
\Dc_{i,i\ell}&\equiv&(-\Da_\ell+\Da_{i,i\ell}-\Da_{j,j\ell}-\Da_{k,k\ell})/2
\end{eqnarray}

These relations imply that in their actions on the Albert algebra, any two of the
three copies of $\so(9)$ intersect in the same $\so(8)$ subalgebra. Hence the dimension
of their sum is $28+8+8+8=52$. We must now prove that this sum is closed
under commutation, in order to complete the proof that it is a Lie algebra.
(Of course, in a sense this is all well-known. The point, however, is to find a proof
which generalises nicely to $\ee_8$, where the results are not already known.)

\subsection{An abstract $\ff_4$ Lie algebra}
In terms of actions, the algebra is spanned by $\Za_a$, $\Zb_a$, $\Zc_a$, and the $\Da_p$
and $\Da_{p,q}$. 
The commutator relations between $\Za$s and $\Da$s  
have all been computed above. Similar relations hold between the $\Zb$s and $\Db$s, and between the
$\Zc$s and the $\Dc$s. Using the computed linear equations 
between the $\Da$s, $\Db$s and $\Dc$s,
we can deduce the commutator relations between the $\Da$s and both the $\Zb$s and $\Zc$s.
It is worth noting the particular cases
\begin{eqnarray}
[\Da_p,\Zb_b]&=&-\Zb_{pb}
\cr
[\Da_p,\Zc_c]&=&-\Zc_{cp}
\end{eqnarray}
This leaves only
the commutators between $\Za$s, $\Zb$s and $\Zc$s to compute.
The computations are non-trivial so we sketch them here.
The following technical lemma will be useful.
\begin{lemma}\label{byc}
If $y\ne b$ and $z\ne c$ are basic octonion units, then
$\overline{b}(yc)=-\overline{y}(bc)$ and $(bz)\overline{c}=-(bc)\overline{z}$.
\end{lemma}
\begin{proof}
We prove the first equation, the second being essentially the same.
First, if $y=1$, so $b\ne 1$, then 
$$\overline{b}(yc)=\overline{b}c=-bc=-\overline{y}(bc).$$
Similarly, if $b=1$. Now we can assume $b\ne 1\ne y$, and either $b,c,y$ associate,
in which case
$$\overline{b}(yc)=(\overline{b}y)c=-(\overline{y}b)c=-\overline{y}(bc),$$
or they anti-associate, in which case only the signs of the middle two expressions change.
\end{proof} 
\begin{lemma}
In the action on the Albert algebra, the relation
$$[\Zb_b,\Zc_c]=-\Za_{\overline{bc}}$$ holds for any octonions $b,c$.
\end{lemma}
\begin{proof}
We first compute the action of $\Za_a$ directly, with special cases
\begin{eqnarray}
\Wa_0&\mapsto& -2\Wa_a\cr
\Wa_a&\mapsto& 2\Wa_0\cr
\Wb_0&\mapsto& \Wc_a\cr
\Wc_0&\mapsto& \Wb_a
\end{eqnarray}
and generic cases
\begin{eqnarray}
\Wa_x&\mapsto &0
\cr
\Wb_y&\mapsto & \Wc_{\overline{ay}}
\cr
\Wc_z&\mapsto & -\Wb_{\overline{za}}
\end{eqnarray}
where $a,b,x,y,z$ all denote basic octonion units, with $x\ne a$.

The actions of $\Zb_b$ and $\Zc_c$ are given by the same formulae with 
the letters $U$, $V$, $W$ cycled in the obvious way. Hence we can compute the action
of $[\Zb_b,\Zc_c]=\Zb_b\Zc_c-\Zc_c\Zb_b$. The special cases are straightforward:
\begin{eqnarray}
\Wa_0&\mapsto& 2\Wa_{\overline{bc}}\cr
\Wb_0&\mapsto& -\Wc_{\overline{bc}}\cr
\Wc_0&\mapsto& -\Wb_{\overline{bc}}\cr
\Wa_{\overline{bc}}&\mapsto&2\Wb_0+2\Wc_0=-2\Wa_0
\end{eqnarray}
but the generic cases require some extra work:
\begin{eqnarray}
\Wa_x&\mapsto& 0\cr
\Wb_y&\mapsto& \Wc_{\overline{b}(yc)} -2\delta_{yb}\Wc_c\cr
\Wc_z&\mapsto& -\Wb_{(bz)\overline{c}}+2\delta_{zc}\Wb_b
\end{eqnarray}
To see this, observe first that
in the case $y=b$ the image of $\Wb_y$ simplifies to $-\Wc_c$, and in the case
$z=c$ the image of $\Wc_z$ simplifies to $\Wb_b$.
The remaining cases are dealt with by Lemma~\ref{byc}.
\end{proof}

\subsection{Matrix commutators}
\label{f4lie} 
We have already seen that the Lie bracket in $\so(9)$ agrees with the matrix commutator,
subject to some quite intuitive rules for extending the octonion multiplication to 
incorporate terms of the form $p\circ q$. In addition, it is easy to check that the commutator
rule $[\Za_a,\Zb_b]=-\Zc_{\overline{ab}}$ also corresponds exactly to the natural matrix commutator.
It follows that the Lie bracket can be completely calculated using matrices, provided
that diagonal matrices are always expressed in the appropriate form for the
calculation being performed.

It is interesting, however, to explore to what extent mixed commutators, such as
$[\Da_i,\Db_j]$ or $[\Da_i,\Zb_j]$, can be computed directly, without first converting the $\Da$s,
$\Db$s or $\Dc$s
into the appropriate form. 
Let us consider first the commutators between one diagonal and one off-diagonal element. If we compute
\begin{eqnarray}
[\Db_\ell,\Za_i]&=&[-\Da_\ell-\Da_{i,i\ell}-\Da_{j,j\ell}-\Da_{k,k\ell},\Za_i]/2\cr
&=&-\Za_{i\ell}
\end{eqnarray}
then we get exactly what we would expect by naively computing the commutator of the
matrices:
\begin{eqnarray}
\left[\begin{pmatrix}0&0&0\cr 0&\ell&0\cr0&0&-\ell\end{pmatrix},
\begin{pmatrix}0&i&0\cr i&0&0\cr 0&0&0\end{pmatrix}\right]
&=&
\begin{pmatrix}
0&-i\ell&0\cr \ell i&0&0\cr 0&0&0
\end{pmatrix}
\end{eqnarray}
By symmetry, this shows that most of the commutators of $\Za$s,
$\Zb$s or $\Zc$s with `single-index' $\Da$s, $\Db$s or $\Dc$s
can be computed as matrix commutators.
The remaining cases $[\Db_p,\Za_1]$ and $[\Dc_p,\Za_1]$ are easy, since
there is sufficient associativity in these cases to identify the
action of the commutator matrix with the commutator of the actions.

Next we investigate the commutators of $\Zb$s
and $\Zc$s with `double-index' $\Da$s.
We first calculate with the actions on the Albert algebra, using the Jacobi identity,
which holds automatically for actions.
\begin{eqnarray}
2[\Zb_a,\Da_{p,q}]&=&[\Zb_a,[\Da_p,\Da_q]]\cr
&=&-[\Da_p,[\Da_q,\Zb_a]]-[\Da_q,[\Zb_a,\Da_p]]\cr
&=&[\Da_p,\Zb_{jqa}]-[\Da_q,\Zb_{pa}]\cr
&=&-\Zb_{p(qa)}+\Zb_{q(pa)}\cr
&=&-2\Zb_{p(qa)}
\end{eqnarray}
Now if we try computing matrix commutators we obtain
\begin{eqnarray}
[\Zb_a,\Da_{p,q}]_M&=&\begin{pmatrix}0&0&0\cr 0&0&a\cr 0&-\overline{a}&0\end{pmatrix}
\begin{pmatrix}p\circ q&0&0\cr0&p\circ q&0\cr0&0&0\end{pmatrix}
\cr && \qquad -
\begin{pmatrix}p\circ q&0&0\cr0&p\circ q&0\cr0&0&0\end{pmatrix}
\begin{pmatrix}0&0&0\cr 0&0&a\cr 0&-\overline{a}&0\end{pmatrix}
\cr
&=&\begin{pmatrix}0&0&0\cr0&0&-(p\circ q)a\cr0&-\overline{a}(p\circ q) &0\end{pmatrix}
\end{eqnarray}
Clearly, if these two methods of computation are to give the
same answer, then we require 
\begin{eqnarray}
(p\circ q)a&:=& p(qa)\cr
a(p\circ q)&:=&(ap)q
\end{eqnarray}
which is again a natural convention.
That is, $p\circ q$ is interpreted as a composition of operators,
acting on left or right as appropriate, and with the innermost operator acting first.
Then the above matrix reduces to
\begin{eqnarray}
\begin{pmatrix}0&0&0\cr 0&0&-p(qa)\cr 0&-(\overline{a}p)q&0
\end{pmatrix}&=& \Zb_{-p(qa)}\cr
&=&\Zb_{q(pa)}
\end{eqnarray}
as required.

Similarly, for $[\Zc_a,\Da_{p,q}]$, again
calculating in the Lie algebra we have constructed, and 
using the Jacobi identity, we have
\begin{eqnarray}
2[\Zc_a,\Da_{p,q}]&=&[\Zc_a,[\Da_p,\Da_q]]\cr
&=&-[\Da_p,[\Da_q,\Zc_a]]-[\Da_q,[\Zc_a,\Da_p]]\cr
&=&[\Da_p,\Zc_{aq}]-[\Da_q,\Zc_{ap}]\cr
&=&-\Zc_{(aq)p}+\Zc_{(ap)q}\cr
&=&2\Zc_{(ap)q}
\end{eqnarray}
Then we compare with the naive computation of commutators of matrices, using the
same interpretations of $(p\circ q)a$ and $a(p\circ q)$ as above. We have
\begin{eqnarray}
[\Zc_a,\Da_{p,q}]_M&=&\begin{pmatrix}0&0&-\overline{a}\cr 0&0&0\cr a&0&0\end{pmatrix}
\begin{pmatrix}p\circ q&0&0\cr0&p\circ q&0\cr0&0&0\end{pmatrix}
\cr&&\qquad -
\begin{pmatrix}p\circ q&0&0\cr0&p\circ q&0\cr0&0&0\end{pmatrix}
\begin{pmatrix}0&0&-\overline{a}\cr 0&0&0\cr a&0&0\end{pmatrix}
\cr
&=&\begin{pmatrix}0&0&(p\circ q)\overline{a}\cr0&0&0\cr{a}(p\circ q) &0&0\end{pmatrix}
\end{eqnarray}
and again we get the correct answer. 

We turn now to the computation of commutators of two diagonal elements.
In the Lie algebra, we have
\begin{eqnarray}
[\Db_\ell,\Da_i]&=&-[\Da_\ell+\Da_{i,i\ell}+\Da_{j,j\ell}+\Da_{k,k\ell},\Da_i]/2\cr
&=&\Da_{i,\ell}-\Da_{i\ell}\cr
&=&\Da_{i,\ell}+\Db_{i,\ell}-\Dc_{i,\ell}
\end{eqnarray}
The same computation with matrices gives
\begin{eqnarray}
[\Db_{\ell},\Da_i]_M&=& \begin{pmatrix}0&0&0\cr 0&2i\circ \ell &0\cr 0&0&0\end{pmatrix}\cr
&=&\Da_{i,\ell}+\Db_{i,\ell}-\Dc_{i,\ell}
\end{eqnarray}
By symmetry, therefore, 
the commutators of single-index $\Da$s, $\Db$s and $\Dc$s
can all be calculated correctly by naive matrix commutator calculations.

Unfortunately, the same is not true for double-index elements.
The easiest way to translate these elements between $\Da$, $\Db$ and $\Dc$ is
probably to use the equations
\begin{eqnarray}
\Db_{p,q} - \Db_{pq} &=& \Dc_{p,q} + \Dc_{pq}
\end{eqnarray}
and their images under rotations of $\Da$, $\Db$, $\Dc$.

To summarise the results of this section, we have shown how to compute all commutators
between elements of the basis consisting of $\Da_p$, $\Da_{p,q}$, $\Za_a$, $\Zb_a$, $\Zc_a$,
in such a way that the matrix commutator, suitably interpreted, 
agrees with the commutator of the action on
the Jordan algebra. We have derived certain triality relations between the 
bases of type $\Da$, $\Db$ and $\Dc$ for $\so(8)$,
which describe relations betwen different linear combinations
of basis vectors which act in the same way on the Albert algebra. 
These relations are used in an essential way in the computation of matrix
commutators, and we thereby obtain
an abstract definition of the Lie algebra of
type $\ff_4$, with respect to any one of three triality-related bases. 

More formally:
 
\begin{theorem}\label{f4thm}
Let $\lie$ be the real $52$-space spanned by  the $3\times 3$ octonionic matrices
$\Za_a$, $\Zb_a$, $\Zc_a$ together with either
$\Da_p$, 
$\Da_{p,q}$, 
or $\Db_p$, $\Db_{p,q}$, or 
$\Dc_p$, $\Dc_{p,q}$
defined above,
subject to the relations given above which identify the three $28$-spaces spanned by
the $\Da$s, the $\Db$s and the $\Dc$s.
Define a commutator on $\lie$ as the matrix commutator, subject to 
the
following special rules:
\begin{enumerate}
\item Commutators of two diagonal matrices must be computed in one of the three
languages $\Da$, $\Db$ or $\Dc$, and not mixed;
\item Commutators within each of the three copies of $\so(9)$ are
computed as in Theorem~\ref{so9thm}.
\item Commutators of any $\Za$ with any $\Zb$ or $\Zc$, or of any $\Zb$ with any $\Zc$,
are computed with the usual octonion multiplication.
\item Commutators of any $\Da$ with any $\Zb$ or $\Zc$ (and similarly for $\Db$ with $\Za$ or $\Zc$,
and for $\Dc$ with $\Za$ or $\Zb$) may be computed by using the rules
$(p\circ q)a=p(qa)$ and $a(p\circ q)=(ap)q$.
\end{enumerate}
Then $\lie$ with this commutator map is isomorphic to the compact
real form of the Lie algebra $\ff_4$, if the
octonion division algebra is used, or to the form 
$\ff_{4(4)}$, if the split octonions are used. 
\end{theorem}

\section{The Lie algebra $\ee_8=\su(3,\OO\otimes\OO)$
}
\label{e8}
\subsection{Introduction}
The minimal
representation of $\ee_8$ is the adjoint representation of the Lie algebra on itself.
A
concrete approach is to build the algebra out of three copies of $\so(16)$,
each acting on an explicitly defined $248$-dimensional vector space. These actions
then generate an associative subalgebra of the full $248\times 248$ matrix algebra.
 We show that
any two of these copies of $\so(16)$ 
intersect in $\so(8)+\so(8)$, and that they therefore
span a space of dimension $248$ inside the matrix algebra.
Finally we verify explicitly that the commutators of these $248$ linear maps
also lie in this $248$-space. Thus the space is closed under commutation, and
since the Jacobi identity holds automatically, we 
identify the algebra as $\ee_8$.

\subsection{A $248$-dimensional space}
We begin by defining the $248$-dimensional vector space.
Let $\KK$ and $\LL$ be two copies of the octonions, with standard basis elements
$1,i,j,k,\ell,i\ell,j\ell,k\ell\in\KK$ and $1,I,J,K,L,IL,JL,KL\in\LL$,
such that 
$ij=k$, $IJ=K$, and $i,j,\ell$ generate $\KK$ and $I,J,L$ generate $\LL$. 
For basis elements $a$ of $\KK$ and $A$ of $\LL$, define
\begin{eqnarray}
\Za_{aA}&=&\begin{pmatrix}0&a\otimes A&0\cr -\overline{a}\otimes \overline{A}&0&0\cr 0&0&0\end{pmatrix}\cr
\Zb_{aA}&=&\begin{pmatrix}0&0&0\cr 0&0& a\otimes A\cr 0 &-\overline{a}\otimes \overline{A}&0\end{pmatrix}\cr
\Zc_{aA}&=&\begin{pmatrix}0&0& -\overline{a}\otimes \overline{A}\cr
0&0&0\cr a\otimes A&0&0\end{pmatrix}
\end{eqnarray}
Under matrix commutation, as defined in Section~\ref{so16}, the matrices $\Za_{aA}$
generate a copy of $\so(16)$, containing the matrices of type
$\Da_{p}$, $\Da_{p,q}$, $\Da_P$ and $\Da_{P,Q}$. These together with the $\Za_{aA}$ form
a $120$-dimensional vector space, on which $\so(16)$ acts
in its adjoint representation. The $\Zb_{aA}$ and $\Zc_{aA}$ form a $128$-dimensional
space, on which $\so(16)$ acts in its spin representation.
Let $\mathcal V$ to be
the $248$-dimensional vector space spanned by the matrices just mentioned.

Similarly we may define spaces $\mathcal V'$ and $\mathcal V''$ on which the other
two copies of $\so(16)$ act. 
Now restricting the subscripts to lie in $\mathbb K$, we must recover the
construction of $\ff_4$ from three copies of $\so(9)$ intersecting in $\so(8)$.
Therefore we must have the same
relations as before expressing the 
$\Db$ and $\Dc$ elements as linear combinations of the $\Da$ type elements. 
The same applies, of course, to the restriction to $\mathbb L$.
These relations are
enough to identify the three vector spaces $\mathcal V$,
$\mathcal V'$ and $\mathcal V''$,
and hence they allow us to construct actions of all three copies of
$\so(16)$ on the same $248$-space.

\subsection{The space acts on itself}
As a vector space, $\mathcal V$ is the sum of three copies of $\so(16)$, any two of
which intersect in $\so(8)\oplus \so(8)$.
The next step is to show that 
every one of
our generators for $\so(16)$
acts on this $248$-space, as a combination of its action on both
the adjoint and spin representations. The argument is much the same as for
$\so(9)$ in $\ff_4$
The basic idea is to use the action of $3\times 3$ anti-Hermitian matrices on each other via
\begin{eqnarray}
A&:& X \mapsto AX-XA.
\end{eqnarray}
But we need to verify that this really is the same as the action of $\so(16)$ on the
sum of the adjoint and spin representations as described in Section~\ref{so16}. Breaking
up the $3\times 3$ matrices in blocks of sizes $2+1$ we can write the diagonal parts
of both $A$ and $X$ so that they have $0$ in the third place. Then,
writing $A'$ and $X'$ for the top-left $2\times 2$ blocks of $A$ and $X$, we have
\begin{eqnarray}
\begin{pmatrix}A'&0\cr0&0\end{pmatrix} &:&
\begin{pmatrix}X'&v\cr -v^\dagger&0\end{pmatrix} \mapsto
\begin{pmatrix}A'X'-X'A' & A'v\cr v^\dagger A' & 0 \end{pmatrix}
\end{eqnarray}
Now we see the adjoint action given by matrix commutation in the $2\times 2$ block,
and the spinor action $A':v\mapsto A'v$ exactly as described in  
Section~\ref{so16}. The bottom row $v^\dagger A'=-v^\dagger A'^\dagger=-(A'v)^\dagger$
just repeats the spinor action.
It follows that our generators for the three copies of $\so(16)$ generate a subalgebra
of the associative algebra of $248\times248$ matrices. Therefore under commutation
they generate a Lie algebra. 

\subsection{Closure}
It suffices now to show that all the commutators are already
in the $248$-space spanned by the generators.
Within $\so(16)$ we already know that the commutator of the actions on $\mathcal V$
is given by the matrix commutator as described in Section~\ref{so16}. In particular,
all such commutators lie in the $248$-space,
which consists of matrices modulo the relations given above.

The only case not already dealt with inside one of the three copies of
$\so(16)$ is the matrix commutator
\begin{eqnarray}
[\Za_{aA},\Zb_{bB}]_M&=& -\Zc_{\overline{abAB}},
\end{eqnarray}
which we must now prove is equal to the commutator of the actions of $\Za_{aA}$ and
$\Zb_{bB}$.

\begin{lemma}
In the action on $\mathcal V$, we have $[\Za_{aA},\Zb_{bB}]= -\Zc_{\overline{abAB}}$
for all basic units $a,b\in\KK$ and $A,B\in\LL$. 
\end{lemma}
\begin{proof}
First consider the case $bB=1$, that is $b=1$ and $B=1$.
If we compute directly the action of
$[\Za_{aA},\Zb_1]$ on any vector in the $248$-space, there are no 
issues with non-associativity, 
and therefore the action is the same as the matrix commutator $[\Za_{aA},\Zb_1]_M
=-\Zc_{\overline{aA}}$. The same applies to the commutators $[\Za_{A},\Zb_{b}]$.

Now the Jacobi identity, which holds automatically for
linear actions on a vector space, allows us to extend this to the general case.
We deduce first that
for $a\ne 1$, $b\ne 1$ and $B=1$
\begin{eqnarray}
[\Za_{aA},\Zb_b] &=& [\Za_{aA},[\Zb_1,\Da_b]]\cr
&=& -[\Zb_1,[\Da_b,\Za_{aA}]]-[\Da_b,[\Za_{aA},\Zb_1]]\cr
&=&-[\Zb_1,\Za_{abA}]-[\Zb_1,\Za_{baA}]+[\Da_b,\Zc_{\overline{aA}}]
\end{eqnarray}
which in the case $b\ne a$ evaluates to
$-\Zc_{\overline{a}b\overline{A}}=-\Zc_{\overline{abA}}$
and in the case $b=a$ evaluates to
$$-2\Zc_{\overline{abA}}-\Zc_{\overline{a}b\overline{A}}=-\Zc_{\overline{abA}}.$$
Similarly
if $B\ne 1$ we have
\begin{eqnarray}
[\Za_{aA},\Zb_{bB}]&=& [\Za_{aA},[\Zb_b,\Da_B]]\cr
&=& -[\Zb_b,[\Da_B,\Za_{aA}]] - [\Da_B,[\Za_{aA},\Zb_b]]\cr
&=& -[\Zb_b,\Za_{aAB}] -[\Zb_b,\Za_{aBA}]- [\Da_B,\Zc_{\overline{abA}}]
\end{eqnarray}
which we again evaluate separately in the cases $B=A$ and $B\ne A$.
This completes the proof.
\end{proof}

\subsection{The main result}
Pulling together all the various strands, we have the following main result.
\begin{theorem}
Let $\KK$ and $\LL$ be two copies of the octonions (either split or compact,
or one of each), with typical elements
$a,p,q\in\KK$ and $
A,P,Q\in\LL$,
such that $a,p,q$ are orthogonal, and $p,q$ are imaginary,
and similarly $A,P,Q$ are orthogonal and $P,Q$ are imaginary.

Define diagonal matrices $\Da_p, \Da_P, \Da_{p,q}, \Da_{P,Q}$
as 
$$\Da_p=\begin{pmatrix}p&0&0\cr 0&-p&0\cr 0&0&0\end{pmatrix}, 
\qquad \Da_{p,q}=\begin{pmatrix}p\circ q&0&0\cr 0&p\circ q&0\cr 0&0&0\end{pmatrix},$$
and similarly for $\LL$.
Define similarly matrices $\Db_p$ and $\Db_{p,q}$ as the images of $\Da_p$
and $\Da_{p,q}$ under rotating
the three coordinates down and to the right, and $\Dc_p$ and $\Dc_{p,q}$
as images under the inverse rotation.

Define off-diagonal matrices with entries in $\KK\otimes\LL$ as 
\begin{eqnarray}
\Za_{aA}&=&\begin{pmatrix}0&a\otimes A&0\cr -\overline{a}\otimes \overline{A}&0&0\cr 0&0&0\end{pmatrix}
\end{eqnarray}
and similarly $\Zb_{aA}$ and $\Zc_{aA}$ by rotating respectively
down-and-right and up-and-left.

Define a real 
vector space $\mathcal V$ 
spanned by all the matrices just defined, by imposing the
linear 
relations given in Section~\ref{f4triality},
that is, if $\KK$ is compact, the relations
\begin{eqnarray}
\Db_\ell&\equiv& -(\Da_\ell+\Da_{i,i\ell}+\Da_{j,j\ell}+\Da_{k,k\ell})/2\cr
\Dc_\ell&\equiv& (-\Da_\ell+\Da_{i,i\ell}+\Da_{j,j\ell}+\Da_{k,k\ell})/2
\end{eqnarray}
together with the images under the $7$-cycle $(i,j,\ell,k,j\ell,-k\ell,i\ell)$,
and \emph{mutatis mutandis} for $\LL$ and for split octonions.
Then $\mathcal V$ has dimension $248$.

Define an antisymmetric product on $\mathcal V$ via
$$[V,W]=VW-WV,$$ where commutators of diagonal matrices
are computed in each copy of $\so(8)$ in any one of the three languages
$\Da,\Db,\Dc$ (but not mixed!), and octonion operations are extended to 
nested octonions via the rules
\begin{eqnarray}
[p\circ q,a]&=&p\circ q\circ a-a\circ p\circ q \quad(=0\mbox{ unless }a=p\mbox{ or }a=q)\cr
(p\circ q)a&=& p(qa)\cr
a(p\circ q)&=&(ap)q.
\end{eqnarray}
Then this product is well-defined, i.e. does not depend on the
representing matrices $V,W$ for elements of $\mathcal V$.
Moreover, this product satisfies the Jacobi identity,
and makes $\mathcal V$ into a copy of the compact real form of the
Lie algebra $\ee_8$ (if both $\KK$ and $\LL$ are division algebras),
or the split real form (if both $\KK$ and $\LL$ are split octonions), or the other
real form $\ee_{8(-24)}$ (if one of each).
\end{theorem}

\end{document}